\documentclass[oneside,reqno,english]{amsart}
\usepackage[latin9]{inputenc}
\usepackage{amstext}
\usepackage{amsthm}
\usepackage{amssymb}
\usepackage{geometry}
\usepackage{setspace}
\usepackage{esint}

\makeatletter
\theoremstyle{plain}
\newtheorem{lem}{\protect\lemmaname}
\theoremstyle{plain}
\newtheorem{thm}{\protect\theoremname}

\@ifundefined{date}{}{\date{}}
\PassOptionsToPackage{reqno}{amsmath}

\usepackage{mathtools}

\makeatletter
\AtBeginDocument{\tagsleft@false} % ensures numbers appear on the right
\makeatother

\usepackage{chngcntr}
\counterwithout{equation}{section}

\usepackage{microtype}   % improves text spacing and kerning
\usepackage{setspace}    % control line spacing
\makeatother

\usepackage{babel}
\providecommand{\lemmaname}{Lemma}
\providecommand{\theoremname}{Theorem}

\begin{document}
\title{Unbounded Oscillation of Euler-Gompertz Diophantine Errors from Bell
and Gould Numbers}
\author{Michael R. Powers}
\address{Department of Finance, School of Economics and Management, Tsinghua
University, Beijing, China 100084}
\email{powers@sem.tsinghua.edu.cn}
\date{June 15, 2026}
\begin{abstract}
We investigate the asymptotic behavior of the Diophantine errors $\delta B_{n}-A_{n}$,
where $\delta=-e\textrm{Ei}\left(-1\right)$ denotes the Euler-Gompertz
constant and $B_{n}$ and $A_{n}$ are the $n$th Bell and Gould numbers,
respectively. These errors have exponential generating function $g_{\delta}\left(z\right)=\exp\left(e^{z}-1\right)\left(\delta-{\displaystyle {\textstyle \int_{0}^{z}}\exp\left(1-e^{t}\right)dt}\right),\:z\in\mathbb{C}$,
and it is known that $\delta B_{n}-A_{n}=O(B_{n}\exp(-cn/\left(\ln\left(n\right)\right)^{2}))$
for some $c\in\mathbb{R}_{>0}$, implying $\lim_{n\rightarrow\infty}\left(A_{n}/B_{n}\right)=\delta$.
In the present work, we prove that $\delta B_{n}-A_{n}$ oscillates
without bound as $n\rightarrow\infty$; that is, both $\limsup_{n\rightarrow\infty}\left(\delta B_{n}-A_{n}\right)=\infty$
and $\liminf_{n\rightarrow\infty}\left(\delta B_{n}-A_{n}\right)=-\infty$.
\end{abstract}

\maketitle

\section{Introduction}

\noindent Let 
\begin{equation}
\left\{ B_{n}\right\} _{n\geq0}=\left\{ 1,1,2,5,15,52,\ldots\right\} 
\end{equation}
 and 
\begin{equation}
\left\{ A_{n}\right\} _{n\geq0}=\left\{ 0,1,1,3,9,31,\ldots\right\} 
\end{equation}
denote the integer-valued sequences of Bell and Gould numbers, respectively.\footnote{The Bell numbers are recorded as OEIS sequence A000110, and the Gould
numbers as OEIS sequence A040027, where the latter sequence is offset
by one term (i.e., A040027 drops the initial $0$).} These sequences have corresponding exponential generating functions
\[
g_{B}\left(z\right)=e^{e^{z}-1}={\displaystyle \sum_{n=0}^{\infty}}\dfrac{B_{n}}{n!}z^{n},\:z\in\mathbb{C},
\]
and
\[
g_{A}\left(z\right)=e^{e^{z}-1}{\displaystyle \int_{0}^{z}e^{1-e^{t}}dt}={\displaystyle \sum_{n=0}^{\infty}}\dfrac{A_{n}}{n!}z^{n},\:z\in\mathbb{C},
\]
where: (a) the function $\exp\left(1-e^{t}\right)$ is entire on $\mathbb{C}$;
(b) the integral ${\textstyle \int_{0}^{z}}\exp\left(1-e^{t}\right)dt$
is path independent; and (c) ${\textstyle \int_{0}^{\infty}}\exp\left(1-e^{t}\right)dt=-e\textrm{Ei}\left(-1\right)=\delta$,
the Euler-Gompertz constant (see Gould and Quaintance {[}1{]}). The
Bell and Gould numbers can be combined to yield the sequence of Diophantine
errors
\begin{equation}
\left\{ \delta B_{n}-A_{n}\right\} _{n\geq0},
\end{equation}
which is of research interest because the irrationality of $\delta$
remains an open question. This sequence has exponential generating
function
\[
g_{\delta}\left(z\right)=\delta g_{B}\left(z\right)-g_{A}\left(z\right)
\]
\begin{equation}
=e^{e^{z}-1}\left(\delta-\int_{0}^{z}e^{1-e^{t}}dt\right)={\displaystyle \sum_{n=0}^{\infty}}\dfrac{\left(\delta B_{n}-A_{n}\right)}{n!}z^{n},\:z\in\mathbb{C},
\end{equation}
which simplifies to $g_{\delta}\left(x\right)=-\exp\left(e^{x}\right)\textrm{Ei}\left(-e^{x}\right)$
for $x\in\mathbb{R}$.

In their study of (1), (2), and (3), Gould and Quaintance {[}1{]}
conjectured that
\begin{equation}
\underset{n\rightarrow\infty}{\lim}\dfrac{A_{n}}{B_{n}}=\delta,
\end{equation}
a result subsequently confirmed by Asakly et al. {[}2{]}. To prove
(5), the latter authors showed that
\begin{equation}
\delta B_{n}-A_{n}=O\left(B_{n}e^{-cn/\left(\ln\left(n\right)\right)^{2}}\right)
\end{equation}
\[
\Longrightarrow\delta-\dfrac{A_{n}}{B_{n}}=O\left(e^{-cn/\left(\ln\left(n\right)\right)^{2}}\right)
\]
for some $c\in\mathbb{R}_{>0}$.

It is interesting to note that the scale factor $B_{n}\exp(-cn/\left(\ln\left(n\right)\right)^{2})$
appearing on the right-hand side of (6) is unbounded as $n\rightarrow\infty$
because
\[
B_{n}\geq\left(\dfrac{n}{e\ln\left(n\right)}\right)^{n}
\]
for all $n\in\mathbb{Z}_{\geq2}$.\footnote{See Grunwald and Serafin {[}3{]}.}
This implies
\[
B_{n}e^{-cn/\left(\ln\left(n\right)\right)^{2}}\geq\left(\dfrac{n}{e\ln\left(n\right)}\right)^{n}e^{-cn/\left(\ln\left(n\right)\right)^{2}}
\]
\[
\Longrightarrow\ln\left(B_{n}e^{-cn/\left(\ln\left(n\right)\right)^{2}}\right)\geq n\left[\ln\left(n\right)-1-\ln\left(\ln\left(n\right)\right)\right]-\dfrac{cn}{\left(\ln\left(n\right)\right)^{2}}
\]
\begin{equation}
=n\left[\ln\left(n\right)-1-\ln\left(\ln\left(n\right)\right)-\dfrac{c}{\left(\ln\left(n\right)\right)^{2}}\right],
\end{equation}
where the right-hand side of (7) diverges to positive infinity as
$n\rightarrow\infty$. However, $B_{n}\exp(-cn/\left(\ln\left(n\right)\right)^{2})\rightarrow\infty$
does not necessarily imply $\left|\delta B_{n}-A_{n}\right|\rightarrow\infty$,
although the latter divergence is supported by informal numerical
computation. Indeed, it is possible (albeit unlikely) that $\left|\delta B_{n_{k}}-A_{n_{k}}\right|\rightarrow0^{+}$
for some subsequence $\left\{ n_{k}\right\} _{k\geq0}$, a fact that
would imply the irrationality of $\delta$ provided $\left|\delta B_{n_{k}}-A_{n_{k}}\right|>0$
for all $k$.

In the present investigation, we show that $\delta B_{n}-A_{n}$ oscillates
without bound as $n\rightarrow\infty$; that is, both $\limsup_{n\rightarrow\infty}\left(\delta B_{n}-A_{n}\right)=\infty$
and $\liminf_{n\rightarrow\infty}\left(\delta B_{n}-A_{n}\right)=-\infty$.
These results follow from two basic lemmas presented in Section 2.
The main theorem is proved in Section 3.

\section{Two Lemmas}

\noindent This section provides two simple lemmas used in the proof
of our main result. The first lemma decomposes $g_{\delta}\left(x+2\pi i\right)$
into its real and imaginary parts and establishes a decay estimate
for $g_{\delta}\left(x\right)$ as $x\rightarrow\infty$. The second
lemma gives a useful asymptotic property of entire functions with
nonnegative real power-series coefficients.
\begin{lem}
\begin{singlespace}
\noindent Let $z=x+2\pi i\in\mathbb{C}$, where $x\in\mathbb{R}_{>0}$.
Then:
\[
\textrm{(i) }g_{\delta}\left(z\right)=g_{\delta}\left(x\right)-2\pi ie^{e^{x}};
\]
and
\[
\textrm{(ii) }g_{\delta}\left(x\right)=O\left(e^{-x}\right)\:as\:x\rightarrow\infty.
\]
\end{singlespace}
\end{lem}
\begin{proof}
\begin{singlespace}
\phantom{}
\end{singlespace}

\begin{singlespace}
\medskip{}
\end{singlespace}

\noindent (i) Let $z=x+2\pi i\in\mathbb{C}$, and note that (4) implies
\[
g_{\delta}\left(x+2\pi i\right)-g_{\delta}\left(x\right)=e^{e^{x+2\pi i}-1}\left(\delta-\int_{0}^{x+2\pi i}e^{1-e^{t}}dt\right)-e^{e^{x}-1}\left(\delta-\int_{0}^{x}e^{1-e^{t}}dt\right)
\]
\begin{equation}
=-e^{e^{x}-1}\int_{x}^{x+2\pi i}e^{1-e^{t}}dt=-e^{e^{x}}\int_{x}^{x+2\pi i}e^{-e^{t}}dt.
\end{equation}
Substituting $u=e^{t}$ into the last integral of (8), where $u$
travels once counterclockwise around the circle $\left|u\right|=e^{x}$
as $t$ moves from $x$ to $x+2\pi i$, we find that
\[
g_{\delta}\left(x+2\pi i\right)-g_{\delta}\left(x\right)=-e^{e^{x}}{\displaystyle \oint}_{\left|u\right|=e^{x}}\dfrac{e^{-u}}{u}du.
\]
Then, since the integrand of the above contour integral has a simple
pole at $0$ with residue $1$, it follows that
\[
g_{\delta}\left(x+2\pi i\right)-g_{\delta}\left(x\right)=-2\pi ie^{e^{x}}
\]
\[
\Longrightarrow g_{\delta}\left(x+2\pi i\right)=g_{\delta}\left(x\right)-2\pi ie^{e^{x}}.
\]

\noindent (ii) From (4), we know that
\[
g_{\delta}\left(x\right)=e^{e^{x}-1}\left(\delta-\int_{0}^{x}e^{1-e^{t}}dt\right)=e^{e^{x}-1}\int_{x}^{\infty}e^{1-e^{t}}dt>0
\]
for $x\in\mathbb{R}_{>0}$. Making the substitution $w=e^{t}-e^{x}$
in the second integral, we then see that
\[
g_{\delta}\left(x\right)=e^{e^{x}-1}\int_{0}^{\infty}\dfrac{e^{1-w-e^{x}}}{w+e^{x}}dw=\int_{0}^{\infty}\dfrac{e^{-w}}{w+e^{x}}dw
\]
\[
\leq e^{-x}\int_{0}^{\infty}e^{-w}dw=e^{-x},
\]
yielding the bounds
\[
0<g_{\delta}\left(x\right)\leq e^{-x}.
\]
\end{proof}
\begin{lem}
\begin{singlespace}
\noindent For $z=x+iy\in\mathbb{C}$, let
\[
G\left(z\right)=\sum_{n=0}^{\infty}a_{n}\dfrac{z^{n}}{n!}
\]
be an entire function with $a_{n}\in\mathbb{R}_{\geq0}$ for all $n\in\mathbb{Z}_{\geq0}$.
If 
\[
G\left(x\right)=O\left(e^{x}\right)
\]
as $x\to\infty$, then
\[
\left|G\left(z\right)\right|=O\left(e^{x}\right)
\]
as $x\to\infty$ for any fixed $y\in\mathbb{R}$.
\end{singlespace}
\end{lem}
\begin{proof}
\begin{singlespace}
\phantom{}
\end{singlespace}

\medskip{}
\noindent Given that $G\left(\cdot\right)$ is entire, the series ${\textstyle \sum_{n=0}^{\infty}}a_{n}\left|z\right|^{n}/n!$
must converge, and it follows from the triangle inequality that
\[
\left|G\left(z\right)\right|\le\sum_{n=0}^{\infty}a_{n}\frac{\left|z\right|^{n}}{n!}=G\left(\left|z\right|\right).
\]
Now assume that
\[
G\left(x\right)=O\left(e^{x}\right)
\]
as $x\to\infty$. We then see that
\[
\left|G\left(z\right)\right|\le G\left(\left|z\right|\right)=G\left(\sqrt{x^{2}+y^{2}}\right)=O\left(e^{\sqrt{x^{2}+y^{2}}}\right),
\]
where
\[
e^{\sqrt{x^{2}+y^{2}}}=e^{x}e^{O\left(y^{2}/x\right)},
\]
implying
\[
\left|G\left(z\right)\right|=O\left(e^{x}\right)
\]
for fixed $y$ as $x\to\infty$.
\end{proof}

\section{Main Result}

\noindent The following theorem gives our central result.
\begin{thm}
Let $\left\{ B_{n}\right\} _{n\geq0}$ and $\left\{ A_{n}\right\} _{n\geq0}$
denote the sequences of Bell and Gould numbers, respectively. Then
\[
\textrm{(i) }\underset{n\rightarrow\infty}{\limsup}\:\left(\delta B_{n}-A_{n}\right)=\infty;
\]
and
\[
\textrm{(ii) }\underset{n\rightarrow\infty}{\liminf}\:\left(\delta B_{n}-A_{n}\right)=-\infty.
\]
\newpage{}
\end{thm}
\begin{proof}
\begin{singlespace}
\phantom{}
\end{singlespace}

\medskip{}
\noindent (i) Assume, for the purpose of contradiction, that $\limsup_{n\rightarrow\infty}\left(\delta B_{n}-A_{n}\right)<\infty$;
that is, there exist constants $N\in\mathbb{Z}_{\geq1}$ and $U\in\mathbb{R}_{>0}$
such that $\delta B_{n}-A_{n}\le U$ for all $n\geq N$. Now select
$V\geq U$, and define both the polynomial
\[
Q^{+}\left(z\right)=-\sum_{n=0}^{N-1}\left[V-\left(\delta B_{n}-A_{n}\right)\right]\frac{z^{n}}{n!}
\]
(of degree at most $N-1$) and the function
\begin{equation}
H^{+}\left(z\right)=Ve^{z}-g_{\delta}\left(z\right)+Q^{+}\left(z\right)
\end{equation}
\begin{equation}
=\sum_{n=N}^{\infty}\left[V-\left(\delta B_{n}-A_{n}\right)\right]\frac{z^{n}}{n!},
\end{equation}
where, for $x\in\mathbb{R}_{>0}$, (9) and (10) yield
\[
H^{+}\left(x\right)=Ve^{x}-g_{\delta}\left(x\right)+Q^{+}\left(x\right)\geq0.
\]
Then, since $Q^{+}\left(x\right)$ is a finite-degree polynomial and
we know from part (ii) of Lemma 1 that $g_{\delta}\left(x\right)=O\left(e^{-x}\right)$,
it follows that
\[
0\leq H^{+}\left(x\right)=O\left(e^{x}\right)+O\left(e^{-x}\right)+o\left(e^{x}\right)=O\left(e^{x}\right)
\]
as $x\rightarrow\infty$.

Given that $H^{+}\left(\cdot\right)$ is entire with nonnegative coefficients
in the expansion of (10), we can fix $y=2\pi$ and conclude from Lemma
2 that
\begin{equation}
\left|H^{+}\left(x+2\pi i\right)\right|=O\left(e^{x}\right).
\end{equation}
On the other hand, part (i) of Lemma 1 allows us to write
\[
H^{+}\left(x+2\pi i\right)=Ve^{x+2\pi i}-g_{\delta}\left(x+2\pi i\right)+Q^{+}\left(x+2\pi i\right)
\]
\[
=Ve^{x}-g_{\delta}\left(x\right)+2\pi ie^{e^{x}}+Q^{+}\left(x+2\pi i\right),
\]
where $Q^{+}\left(x+2\pi i\right)$ is a finite-degree polynomial
in $x$. Thus,
\[
\left|H^{+}\left(x+2\pi i\right)\right|=\sqrt{\left[Ve^{x}-g_{\delta}\left(x\right)+\textrm{Re}\left(Q^{+}\left(x+2\pi i\right)\right)\right]^{2}+\left[2\pi e^{e^{x}}+\textrm{Im}\left(Q^{+}\left(x+2\pi i\right)\right)\right]^{2}}
\]
\[
\geq\left|2\pi e^{e^{x}}+\textrm{Im}\left(Q^{+}\left(x+2\pi i\right)\right)\right|
\]
\begin{equation}
=2\pi e^{e^{x}}+O\left(x^{N-1}\right)\sim2\pi e^{e^{x}}.
\end{equation}
We then see that (12) contradicts (11), implying that the original
hypothesis (i.e., $\limsup_{n\rightarrow\infty}\left(\delta B_{n}-A_{n}\right)<\infty$)
must be false.\medskip{}

\noindent (ii) As in the proof of part (i), we begin with a contradiction
hypothesis. In this case, it is assumed that $\liminf_{n\rightarrow\infty}\left(\delta B_{n}-A_{n}\right)>-\infty$;
that is, there exist constants $N\in\mathbb{Z}_{\geq1}$ and $U\in\mathbb{R}_{>0}$
such that 
\[
\delta B_{n}-A_{n}\geq-U
\]
for all $n\geq N$. We then select $V\geq U$, and define both the
polynomial 
\[
Q^{-}\left(z\right)=-\sum_{n=0}^{N-1}\left[V+\left(\delta B_{n}-A_{n}\right)\right]\frac{z^{n}}{n!}
\]
(of degree at most $N-1$) and the function
\begin{equation}
H^{-}\left(z\right)=Ve^{z}+g_{\delta}\left(z\right)+Q^{-}\left(z\right)
\end{equation}
\begin{equation}
=\sum_{n=N}^{\infty}\left[V+\left(\delta B_{n}-A_{n}\right)\right]\frac{z^{n}}{n!},
\end{equation}
where, for $x\in\mathbb{R}_{>0}$, (13) and (14) yield
\[
H^{-}\left(x\right)=Ve^{x}+g_{\delta}\left(x\right)+Q^{-}\left(x\right)\geq0.
\]
Then, since $Q^{-}\left(x\right)$ is a finite-degree polynomial and
$g_{\delta}\left(x\right)=O\left(e^{-x}\right)$, it follows that
\[
0\leq H^{-}\left(x\right)=O\left(e^{x}\right)+O\left(e^{-x}\right)+o\left(e^{x}\right)=O\left(e^{x}\right)
\]
as $x\rightarrow\infty$.

Given that $H^{-}\left(\cdot\right)$ is entire with nonnegative coefficients
in the expansion of (14), we know from Lemma 2 (with $y=2\pi$) that
\begin{equation}
\left|H^{-}\left(x+2\pi i\right)\right|=O\left(e^{x}\right).
\end{equation}
However, we also know from part (i) of Lemma 1 that
\[
H^{-}\left(x+2\pi i\right)=Ve^{x+2\pi i}+g_{\delta}\left(x+2\pi i\right)+Q^{-}\left(x+2\pi i\right)
\]
\[
=Ve^{x}+g_{\delta}\left(x\right)-2\pi ie^{e^{x}}+Q^{-}\left(x+2\pi i\right),
\]
where $Q^{-}\left(x+2\pi i\right)$ is a finite-degree polynomial
in $x$, implying
\[
\left|H^{-}\left(x+2\pi i\right)\right|=\sqrt{\left[Ve^{x}+g_{\delta}\left(x\right)+\textrm{Re}\left(Q^{-}\left(x+2\pi i\right)\right)\right]^{2}+\left[-2\pi e^{e^{x}}+\textrm{Im}\left(Q^{-}\left(x+2\pi i\right)\right)\right]^{2}}
\]
\[
\geq\left|-2\pi e^{e^{x}}+\textrm{Im}\left(Q^{-}\left(x+2\pi i\right)\right)\right|
\]
\begin{equation}
=2\pi e^{e^{x}}+O\left(x^{N-1}\right)\sim2\pi e^{e^{x}}.
\end{equation}
Since (16) contradicts (15), we conclude that the original hypothesis
(i.e., $\liminf_{n\rightarrow\infty}\left(\delta B_{n}-A_{n}\right)>-\infty$)
is false.
\end{proof}

\section{Conclusion}

\noindent In the present study, we considered asymptotic properties
of the Diophantine errors $\delta B_{n}-A_{n}$, where $\delta$ denotes
the Euler-Gompertz constant and $\left\{ B_{n}\right\} _{n\geq0}$
and $\left\{ A_{n}\right\} _{n\geq0}$ are the sequences of Bell and
Gould numbers, respectively. Previously, Asakly et al. {[}2{]} showed
that $\delta B_{n}-A_{n}=O(B_{n}\exp(-cn/\left(\ln\left(n\right)\right)^{2}))$
for some $c\in\mathbb{R}_{>0}$. We have complemented this estimate
by proving that both $\limsup_{n\rightarrow\infty}\left(\delta B_{n}-A_{n}\right)=\infty$
and $\liminf_{n\rightarrow\infty}\left(\delta B_{n}-A_{n}\right)=-\infty$.

It is important to note that Theorem 1 does not imply $\left|\delta B_{n}-A_{n}\right|\rightarrow\infty$,
although it lends credibility to this assertion, which is supported
by informal numerical computation. Moreover, the theorem does not
preclude the (unlikely) possibility that $\left|\delta B_{n_{k}}-A_{n_{k}}\right|\rightarrow0^{+}$
for some subsequence $\left\{ n_{k}\right\} _{k\geq0}$ -- a fact
that would imply the irrationality of $\delta$ provided $\left|\delta B_{n_{k}}-A_{n_{k}}\right|>0$
for all $k$. We further note that our results do not estimate the
asymptotic size or phase of the oscillations in the Diophantine error.
These refinements will be addressed in subsequent research.

\medskip{}


\begin{thebibliography}{1}
\begin{singlespace}
\bibitem{key-1}Gould, H. W. and Quaintance, J., 2007, ``A Linear
Binomial Recurrence and the Bell Numbers and Polynomials'', \emph{Applicable
Analysis and Discrete Mathematics}, 1, 371-385.

\bibitem{key-2}Asakly, W., Blecher, A., Brennan, C., Knopfmacher,
A., Mansour, T., and Wagner, S., 2014, ``Set Partition Asymptotics
and a Conjecture of Gould and Quaintance'', \emph{Journal of Mathematical
Analysis and Applications}, 416, 672-682.

\bibitem{key-3}Grunwald, J. and Serafin, G., 2025, ``Explicit Bounds
for Bell Numbers and Their Ratios'', \emph{Journal of Mathematical
Analysis and Applications}, 549, 129527.
\end{singlespace}

\end{thebibliography}
\end{document}